\documentclass[11pt]{article}

\usepackage[dvipsnames]{xcolor}

\usepackage{tikz}

\usepackage{cite} 
\usepackage{setspace}
\usepackage[latin1]{inputenc}
\usepackage{amsmath}
 \usepackage{amsthm}
\usepackage{subfig}
\usepackage{amsfonts}
\usepackage{amssymb}
\usepackage{graphicx}
\usepackage{epsfig}
\usepackage{url}
\usepackage{todo}
\usepackage[margin=1.55in]{geometry}
\usepackage{bbm}
\usepackage{microtype}

\usepackage{tikz}
\usepackage[dvipsnames]{xcolor}
\usetikzlibrary{cd}
\usepackage{pgfplots}

\newtheorem{theorem}{Theorem}
\newtheorem{proposition}{Proposition}
\newtheorem{corollary}[theorem]{Corollary}
\newtheorem{lemma}{Lemma}

\newcommand{\R}{\mathbb{R}}

\newcommand{\U}{\mathrm{Uni}}

\newcommand{\conv}{\operatorname{conv}}
\newcommand{\cX}{\mathcal{X}}

\newcommand{\er}{Erd\H{o}s-R\'enyi }

\newcommand{\rint}{\operatorname{int}}

\usepackage{tikz}
\usetikzlibrary{shapes,arrows,graphs,graphs.standard,shapes.misc}
\usetikzlibrary{matrix,decorations.pathreplacing, calc, positioning,fit}
\usetikzlibrary{shapes.geometric}
\usetikzlibrary{decorations.markings}

\newtheorem{definition}{Definition}

\makeatletter
\def\blfootnote{\xdef\@thefnmark{}\@footnotetext}
\makeatother

\def\BibTeX{{\rm B\kern-.05em{\sc i\kern-.025em b}\kern-.08em
    T\kern-.1667em\lower.7ex\hbox{E}\kern-.125emX}}

\begin{document}
\title{The $H$-property of Line Graphons}
\author{M.-A. Belabbas\thanks{M.-A. Belabbas and T.~Ba\c{s}ar are with the Coordinated Science Laboratory, University of Illinois, Urbana-Champaign. Email: \texttt{\{belabbas,basar1\}@illinois.edu.}}
, X. Chen\thanks{X. Chen is with the Department of Electrical, Computer, and Energy Engineering, University of Colorado Boulder. Email: \texttt{xudong.chen@colorado.edu}.}, T.~Ba\c{s}ar$^*$
}
\date{}

\maketitle
\thispagestyle{empty} 

\begin{abstract}
\blfootnote{M.-A. Belabbas and X. Chen contributed equally to the manuscript in all categories.}
We explore in this paper sufficient conditions for the $H$-property to hold, with a particular focus on  the so-called line graphons. A graphon is a symmetric, measurable function from the unit square $[0,1]^2$ to the closed interval $[0,1]$. Graphons can be used to sample random graphs, and a graphon is said to have the $H$-property if graphs on $n$ nodes sampled from it admit a node-cover by disjoint cycles---such a cover is called  a Hamiltonian decomposition---almost surely as $n \to \infty$. A step-graphon is a graphon which is piecewise constant over rectangles in the domain. To a step-graphon, we assign two objects: its concentration vector, encoding the areas of the rectangles, and its skeleton-graph, describing their supports. These two objects were used in~\cite{bcb2021h} to establish necessary conditions for a step-graphon to have the $H$-property. In this paper, we prove that these conditions  are essentially also sufficient for the class of line-graphons, i.e., the step-graphons whose skeleton graphs are line graphs with a self-loop at an ending node. We also investigate borderline cases where neither the necessary nor the sufficient conditions are  met. 
\end{abstract}

\section{Introduction}\label{sec:introduction}
Graphons, introduced in~\cite{lovasz2006limits, borgs2008convergent} to study very large graphs, are increasingly relied upon as models for large networks. 
Roughly speaking, a graphon is a symmetric, measurable function $W: [0,1]^2\to [0,1]$ which can be thought of as an infinite-dimensional adjacency matrix.
Graphons have been put to use in the statistical analysis of random graphs, where the problem of graphon identification from sample networks~\cite{wolfe2013nonparametric} and the problem of detection of clusters in networks~\cite{choi2017co}, among others, have been explored. 
In parallel, graphons have appeared as models in control and game theory. 
For example, in~\cite{gao2019graphon}, 
the authors considered infinite-dimensional linear control systems where the system ``matrices'' (more precisely, operators on $\mathrm{L}^2([0,1],\R)$) are derived from graphons, and investigated the associated controllability properties and finite-dimensional approximations.
We also mention~\cite{gao2021linear,parise2021analysis} where the authors introduced  different types of graphon games; broadly speaking, these are the games that comprise a continuum of agents and for which the relations between these agents are described by a graphon. They then proceeded to investigate, among others, the existence of Nash equilibria and properties of finite-dimensional approximations.

In the above mentioned works~\cite{gao2019graphon,gao2021linear,parise2021analysis}, graphons were treated as infinite-dimensional extensions of finite dimensional adjacency matrices. We take a different point of view in this paper: we treat graphons as stochastic models for sampling large graphs. 
We follow the research line initiated in our earlier work~\cite{bcb2021h} and investigate the so-called $H$-property (see Definition~\ref{def:Hproperty} below) for graphons. More specifically, in the earlier work, we provided a set of necessary conditions for the $H$-property to hold. In the same paper, it was claimed that these necessary conditions were also essentially sufficient. In this paper, we elaborate on this sufficiency claim and prove it for the class of line-graphons, introduced formally in Subsection~\ref{ssec:linegraphon}. 
The reason for the choice of the class of line graphons is twofold: firstly, graphs sampled from line graphons are common in practical situation, as they encode a simple line topology. 
Secondly, the calculations for this class of graphons can be made rather explicit. Indeed, relying on the form of the necessary conditions, checking for their sufficiency in these cases will only require elementary results from the theory of \er random graphs. This relative simplicity makes an intuitive understanding of these conditions easier to attain.
In an upcoming paper, we will prove the sufficiency of these two conditions in the general case, which will then rely on a more abstract approach.

The remainder of the paper is organized as follows: In Section~\ref{sec:Hproperty}, we review the  procedure to sample random graphs from graphons and reproduce from~\cite{bcb2021h} the definition of $H$-property. Next, in Section~\ref{sec:stepgrafonobjects}, we will first introduce step-graphons and the associated key objects, namely, concentration vector, skeleton graph, and the edge polytope derived from the skeleton graph. In the same section, we will also state the conditions that are necessary or sufficient for a step-graphon to have the $H$-property.  
Then, in Section~\ref{sec:casestudies}, we will establish the sufficiency claim for the class of line graphons and investigate a borderline case where neither the necessary nor the sufficient conditions are met. The paper ends with conclusions.

\section{The $H$-property}\label{sec:Hproperty}

We  start this section by describing how to sample graphs from a graphon. 
\vspace{.2cm}

\noindent 
{\bf Sampling procedure}: Let $\U[0,1]$ be the uniform distribution on $[0,1]$. Given a graphon $W:[0,1]^2 \to [0,1]$, we sample an {\em undirected} graph $G_n=(V,E) \sim W$ on  $n$ nodes from $W$ according to the following the procedure: 
\begin{enumerate}
    \item Sample $y_1,\ldots,y_n\sim \U[0,1]$ independently. 
    We call $y_i$ the {\em coordinate of node} $v_i\in V$. 

    \item For any two distinct nodes $v_i$ and $v_j$, place an edge $(v_i,v_j) \in E$ with probability $W(y_i,y_j)$.
\end{enumerate}
According to the model, the probability of having an edge between nodes $v_i$ and $v_j$ in $G_n$ is thus a Bernoulli random variable with {\em coordinate dependent} parameter $W(y_i,y_j)$.  If the graphon is constant, say $W(s,t)=p$ for all $(s,t)\in [0,1]^2$, then $G_n \sim W$ is an Erd\H{o}s-R\'enyi random graph with parameter~$p$; see Subsection~\ref{ssec:er} for a definition of this class of graphs.
One can thus think, in this context, of graphons as a means to allow for an inhomogeneous probability of existence of an edge. 
\vspace{.2cm}

\noindent
{\bf $H$-property:}  
We next recall the $H$-property introduced in~\cite{bcb2021h}. To do so, we let $W$ be a graphon and $G_n \sim W$. We then define the so-called {\em directed} version $\vec G_n=(V,\vec E)$ of an undirected graph $G_n=(V,E)$, which is obtained by replacing every undirected edge of $G_n$ with two directed edges. 
More precisely, the node set of $\vec G_n$ is the same as the one of $G_n$, and the edge set is given by 
$$\vec E :=\{v_iv_j, v_jv_i \mid (v_i,v_j) \in E \},$$ where, by convention, an undirected edge between $v_i$ and $v_j$ is written as $(v_i,v_j)$ and a directed edge from node $v_i$ to node $v_j$ as $v_iv_j$. 

A {\em Hamiltonian decomposition} in $\vec G_n$  is a subgraph $\vec H =  (V, \vec E')$, with the same node set of $\vec G_n$ such that $\vec H$ is a disjoint union of directed cycles. 
Hamiltonian decompositions appear in various guises in  control problems. We just mention here that they arise in the study of structural stability of linear systems~\cite{belabbas_algorithmsparse_2013,belabbas2013sparse} and structural controllability of linear ensemble systems~\cite{chen2021sparse}.  We refer to~\cite{bcb2021h} for more details.

We now  define the $H$-property:

\begin{definition}[$H$-property~\cite{bcb2021h}]\label{def:Hproperty}
Let $W$ be a graphon and $G_n\sim W$. 
Then, $W$ has the {\bf $H$-property} if 
$$
\lim_{n\to\infty}\mathbb{P}(\vec G_n \mbox{ has a Hamiltonian decomposition}) = 1.
$$
\end{definition}

It turns out that the $H$-property is essentially a ``zero-one'' property:  for {\em almost all} graphons $W$, the probability of having a Hamiltonian decomposition is either $0$ or $1$ in the limit. This property is, however, not true for {\em all} graphons; we provide in Subsection~\ref{ssec:notonezero} an example showcasing this fact. We will elaborate on this property later in the next section.

\section{Step-graphons and associated objects}\label{sec:stepgrafonobjects}

\subsection{Step-graphons}\label{ssec:stepgraphons}
Following~\cite{bcb2021h}, we  restrict our attention to the so-called {\em step-graphons}, defined as follows:

\begin{definition}[Step-graphon and its partition]\label{def:stepgraphon}
We call a graphon $W$ a {\bf step-graphon} if there exists an increasing sequence $0 = \sigma_0 < \sigma_1< \cdots < \sigma_q = 1$ such that $W$ is constant over each rectangle $[\sigma_{i}, \sigma_{i + 1})\times [\sigma_{j}, \sigma_{j + 1})$ for all $0\leq i, j\leq q-1$ (there are $q^2$ rectangles in total).  The sequence $ \sigma = (\sigma_0,\sigma_1,\ldots,\sigma_q)$ is a {\bf partition for $W$}.
\end{definition}

In words,  $W$ is a step-graphon if the interval $[0,1]$ can be split into  subintervals $\mathcal{R}_1,\ldots,\mathcal{R}_q$ with the property that $W$ is constant over their products  $\mathcal{R}_i\times \mathcal{R}_j$, which are rectangles in the plane.

Given a graph $G_n$ sampled from a step-graphon $W$ with partition sequence $\sigma$, we let $n_i(G_n)$ be the number of nodes $v_j$ of $G_n$ whose coordinates $y_j \in [\sigma_{i-1},\sigma_i)$ (see item~1 of the sampling procedure). When $G_n$ is clear from the context, we simply write $n_i$.

\subsection{Concentration vectors and skeleton graphs}\label{ssec:concenandskel}
We now present the key objects associated with a step-graphon that are needed to decide whether it has the $H$-property, namely, its concentration vector, skeleton graph, and the so-called edge polytope of the skeleton graph. These objects were introduced in~\cite{bcb2021h}.    

We first have the following definition:

\begin{definition}[Concentration vector]
    Let $W$ be  a {step-graphon} with partition $\sigma = (\sigma_0,\ldots,\sigma_q)$. 
    The associated {\bf concentration vector} $x^* = (x^*_1,\ldots, x^*_q)$ has entries  defined as follows: 
    $x^*_i := \sigma_i - \sigma_{i-1}$, for all $i = 1,\ldots, q$. 
    The  {\bf empirical concentration vector} of a graph $G_n \sim W$ is defined as
\begin{equation}\label{eq:defecv}x(G_n):= \frac{1}{n}(n_1(G_n),\ldots, n_q(G_n)).
\end{equation}
\end{definition}

When $G_n$ is clear from the context, we will simply use $x$ to denote the empirical concentration vector. 
Observe that for $n$ fixed, $nx = (n_1,\ldots, n_q)$ is a multinomial random variable with $n$ trials and $q$ outcomes with probabilities $x^*_i$, for $1 \leq i \leq q$. From Chebyshev's inequality, we have that for any $\epsilon>0$, 
\begin{equation}\label{eq:oldlemma1}\mathbb{P}(\|x(G_n)-x^*\| > \epsilon)\leq \frac{c}{n^2\epsilon^2},
\end{equation} 
where $c$ is a constant independent of~$\epsilon$ and~$n$.

We next have the following definition:

\begin{definition}[Skeleton graph]\label{def:skeleton}
Let $W$ be a step-graphon with a partition $\sigma = (\sigma_0,\ldots, \sigma_q)$. We define the  undirected graph $S = (U, F)$ on $q$ nodes, called  the {\bf skeleton graph} of $W$ for the partition~$\sigma$,  with $U =\{u_1,\ldots, u_q\}$ and edge set $F$ as follows: there is an edge between $u_i$ and $u_j$ if and only if $W$ is non-zero over $[\sigma_{i-1},\sigma_i)\times [\sigma_{j - 1}, \sigma_j)$. 
\end{definition}

Note that there is a graph homomorphism which assigns the nodes of an arbitrary $G_n = (V, E)\sim W$ to their corresponding  nodes in the skeleton graph $S$: 
\begin{equation}\label{eq:defpi}
\pi:  v_j\in V \mapsto \pi(v_j) = u_i\in U,
\end{equation}
where $u_i$ is such that $\sigma_{i-1} \leq y_j < \sigma_{i}$, with $y_j$ the coordinate of $v_j$.

Let $S=(U,F)$ be a skeleton graph. 
We decompose the edge set of $S$ as $F=F_0 \cup F_1$, where elements of $F_0$ are self-loops, and elements of $F_1$ are edges between distinct nodes. 
Given an arbitrary ordering of its edges and self-loops, we let $Z =[z_{ij}]$ be the associated {\em incidence matrix}, defined as the $|U| \times |F|$ matrix with entries:
\begin{equation}\label{eq:defZS}
z_{ij} := \frac{1}{2}
\begin{cases}
    2, & \text{if } f_j\in F_0 \text{ is a loop on node } u_i,       \\
    1, & \text{if node } u_i \text{ is incident to } f_j\in F_1, \\
    0, & \text{otherwise}.
\end{cases}
\end{equation}
Note that the columns of $Z$ are probability vectors. We now introduce the edge polytope:

\begin{definition}[Edge polytope~\cite{ohsugi1998normal}]\label{def:edgepolytope}
Let $S = (U,F)$ be a skeleton graph and $Z$ be the associated incidence matrix. Let $z_j$, for $1\leq j \leq |F|$, be the columns of $Z$.  
The {\em edge polytope} of $S$, denoted by $\cX(S)$, is the finitely generated convex hull: 
\begin{equation}\label{eq:defXS}
\cX(S):= \conv\{z_j \mid j = 1,\ldots, |F|\}.
\end{equation}
\end{definition}

It is known~\cite{ohsugi1998normal} that if a connected $S$ has an odd cycle (i.e., a cycle of odd length including a self-loop), then the rank of $\cX(S)$ is full, i.e., $(q-1)$. Otherwise, the rank of $\cX(S)$ is $(q-2)$.   

\subsection{Conditions for the $H$-property}\label{ssec:condHprop}
We start by introducing a set of conditions which will be critical for deciding whether or not a step-graphon $W$ has the $H$-property. 
Let $\sigma$ be a partition for $W$, and let $S$ and $x^*$ be the associated skeleton graph and concentration vector. 
For simplicity, we assume in the sequel that $S$ is connected (in general, one needs to apply the conditions below for each connected component of $S$). We state here without a proof that if $S$ is connected, then $G_n\sim W$ is also connected almost surely as $ n \to \infty$.

We now state the conditions:    

\vspace{.2cm}
\noindent {\bf Condition 1:} The graph $S$ has an odd cycle.

\vspace{.2cm}
\noindent {\bf Condition 2A:} The vector $x^*$ belongs to the edge polytope of $S$, i.e.,  $x^* \in \cX(S)$.

\vspace{.2cm}
\noindent {\bf Condition 2B:} The vector $x^*$ belongs to the {\em relative interior} of the edge polytope of $S$, i.e.,  $x^* \in \rint \cX(S)$.
\vspace{.2cm}

The following result has been established in~\cite{bcb2021h}: 

\begin{theorem}\label{thm:main}
Let $W$ be a step-graphon with $\sigma$ a partition.  
Let $S$ and $x^*$ be the associated (connected) skeleton graph   and concentration vector, respectively. 
Let $G_n\sim W$ and $\vec G_n$ be the directed version of $G_n$. If {\em either} Condition 1 {\em or} Condition 2A is {\em not} satisfied, then 
 \begin{equation}\label{eq:nonHproperty}
 \lim_{n\to\infty}\mathbb{P}(\vec G_n \mbox{  has a Hamiltonian decomposition}) = 0.
 \end{equation}
\end{theorem}

We also claimed in~\cite{bcb2021h} that if both Condition 1 and Condition 2B are satisfied, then 
\begin{equation}\label{eq:Hproperty}
 \lim_{n\to\infty}\mathbb{P}(\vec G_n \mbox{  has a Hamiltonian decomposition}) = 1.
\end{equation}
A proof of this fact will be provided in a future publication, but we illustrate it in Section~\ref{sec:casestudies} for the special case where $W$ is a line graphon. We also point out that the borderline case between Condition 2A and Condition 2B, i.e.  when $x^* \in \cX(S)$ but $x^* \notin \rint \cX(S)$, is precisely the one for which the $H$-property is {\em not} a zero-one property.

\section{The $H$-property for line graphons}\label{sec:casestudies}
In this section, we will focus on a special case, namely, step-graphons whose skeleton graphs are line graphs (with a self-loop at one of the ending nodes). To carry out analysis, we need some preliminaries about \er random graphs.

\subsection{On Erd\H{o}s-R\'enyi graphs}\label{ssec:er}

An Erd\H{o}s-R\'enyi random graph $R(n,p)=(V,E)$ on $n$  nodes $V = \{v_1,\ldots,v_n\}$ with parameter $p\in [0,1]$ is a random graph obtained as follows: 
The existences of edges between pairs of distinct nodes are independent, identically distributed Bernoulli random variables with parameter $p$, i.e., $$\mathbb{P}((v_i,v_j) \in E) = p \mbox{ for all } 1 \leq i<j\leq n.$$
We first have the following elementary result:

\begin{lemma}\label{lem:almostsuretriangle}
Let $R(n,p)$ be an Erd\H{o}s-R\'enyi random graph with $p>0$. Then, $R(n,p)$ contains a triangle (which is a complete graph on three nodes without self-loops) almost surely as $n \to \infty$.
\end{lemma}

\begin{proof}
Denote by ${\cal K}$ the event that $R(n,p)$ contains at least one triangle, and by $\bar {\cal K}$ the complementary  event that  it contains {\em no} triangles; clearly, $\mathbb{P}(\bar {\cal K}) = 1 - \mathbb{P}( {\cal K})$. 
Furthermore, denote by $\bar {\cal F}$ the event that $R(n,p)$ is such {\em no} triple of consecutive nodes $(v_{3i+1},v_{3i+2},v_{3i+3})$, for $0\leq i  \leq \lfloor n/3 \rfloor -1$, is  a triangle. Observe that if $R(n,p)$ contains no triangle, then obviously consecutive triples of nodes cannot be triangles, i.e., $\bar {\cal K} \subseteq \bar {\cal F}$.  
Now, since the presence of each individual edge in $R(n,p)$ is an independent event, the probability that $(v_{3i+1},v_{3i+2},v_{3i+3})$ does not form a triangle is $(1-p^3)$. Relying again on the independence, we see that this probability is the same  for every triple $(v_{3i+1},v_{3i+2},v_{3i+3})$. 
Since these triples are pairwise disjoint, the events that they form triangles are also independent of each other. Thus, the probability of the event $\bar {\cal F}$ is $(1-p^3)^{\lfloor n/3 \rfloor}$. Since $p>0$, this probability vanishes as $n \to \infty$. Consequently, $\mathbb{P}(\bar {\cal K}) \leq \mathbb{P}(\bar {\cal F}) \to 0$ and $\mathbb{P}({\cal K}) \to 1$. This completes the proof. 
\end{proof}

We next recall that a {\em bipartite graph}~\cite{diestel2012graph} $B = (V, E)$ is an undirected graph whose node set $V$ admits a partition into two disjoint subsets $V_L$ and $V_R$ such that nodes in $V_L$ (resp. $V_R$) have no edge between them.  

A {\em perfect matching} $P$ in a bipartite graph is a subset of its edge set so that each node is incident to {\em exactly one} edge in the subset $P$ (if a perfect matching exists, then it is necessary that $|V_L|=|V_R| = |P|$). 

When $|V_L|\leq |V_R|$, we define a {\em left-perfect matching}  as a subset of $|V_L|$ edges that are incident to all nodes in $V_L$ and so that each node in $V_R$ is incident to {\em at most} one edge. 

Note that a perfect matching $P$ in a bipartite graph $B$ with $|V_L| = |V_R| = n$ gives rise to a Hamiltonian decomposition in $\vec B$, the directed version of $B$. Indeed,  the two oppositely directed edges that replace an edge in $P$ form a two-cycle in $\vec B$. Since $P$ is a perfect matching, these two-cycles are pairwise disjoint and, moreover, cover all of the nodes in $\vec B$.   
 
We can easily adapt the notion of  \er random graphs to the class of bipartite graphs. Specifically, an Erd\H{o}s-R\'enyi random bipartite graph, denoted by $B(n,m,p)$,  with $|V_L|=n$ and $|V_R|=m$, has an edge set obtained as follows: 
The probability of having an edge between any node in $V_L$ and any node in $V_R$ is $p$, and the events of having such edges are independent.
     
We have the following fact as a corollary of Erd\H{o}s and R\'enyi~\cite[Theorem~2]{erdos1964random} and its proof is omitted due to space limitation.

\begin{lemma}\label{lem:erdosperfectmatching} Let $p \in (0,1]$ be a constant and $B(n,m,p)$ be a random bipartite graph, with $n\leq m$. Then, the probability that $B(n,m,p)$ contains a left-perfect matching is one as $n \to \infty$. 
\end{lemma}

The following result is then a corollary of Lemmas~\ref{lem:almostsuretriangle} and~\ref{lem:erdosperfectmatching}:

\begin{corollary}\label{cor:erhamiltoniandec}
Let $R(n,p) = (V, E)$ be an Erd\H{o}s-R\'enyi random graph with $p>0$. Then, $\vec R(n,p)$ contains a Hamiltonian decomposition almost surely as $n \to \infty$.
\end{corollary}

\begin{proof}
Consider the following two cases for the parity of $n$: 
\vspace{.1cm}

\noindent
{\em Case 1: $n$ is even.} In this case, splitting the node set of $R(n,p)$ arbitrarily into two subsets of cardinality $n/2$, we see that $R(n,p)$  contains an \er random bipartite graph $B(n/2, n/2, p)$ as a subgraph on the same node set $V$. Thus, by Lemma~\ref{lem:erdosperfectmatching}, $R(n,p)$ contains a perfect matching $P$ almost surely as $n\to\infty$. Replacing every edge in $P$ with two oppositely directed edges, we obtain a Hamiltonian decomposition $\vec P$ in $\vec R(n,p)$. 

\vspace{.1cm}
\noindent
{\em Case 2: $n$ is odd.} This case is slightly more complicated as there does not exist a perfect matching that covers all the nodes of $R(n,p)$. 
To resolve the issue, we take a two-step approach: (1) By Lemma~\ref{lem:almostsuretriangle}, we know that $R(n,p)$ contains a triangle~$K_3$ as a subgraph almost surely as $n\to\infty$; (2)  The subgraph $R'$ of $R(n,p)$ induced by the nodes that are {\em not} in the triangle thus has an even number $(n - 3)$ of nodes. 
Using the same arguments as for Case 1, we have that $R'$ has a perfect matching $P'$ almost surely as $n\to\infty$. In this way, the triangle $K_3$ and the matching $P'$ are disjoint and, together, they  cover all the nodes of $R(n,p)$. Moving from undirected to pairs of oppositely directed edges as done in case~1, we have that $\vec R(n,p)$ admits a Hamiltonian decomposition,  formed by a directed triangle in $\vec K_3$ and all the two cycles in $\vec P'$. 
\end{proof}

\begin{figure}[t]
    \centering
    \subfloat[\label{sfig1:stepgraphon}]{
\includegraphics{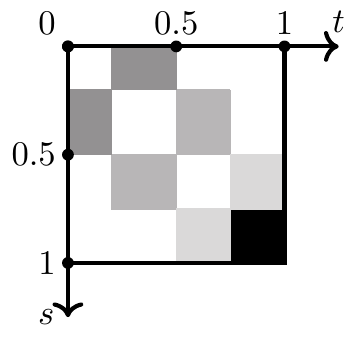}
}
\qquad
\subfloat[\label{sfig1:skeleton}]{
 \includegraphics{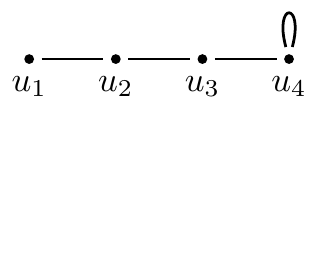}
}
\caption{ {\em Left:} A step-graphon $W$ with the partition $\sigma=(0,0.2,0.5,0.75,1)$. {\em Right:} The associated skeleton graph $S$.}
    \label{fig:linegraphon}
\end{figure}

\subsection{Line Graphons}\label{ssec:linegraphon}
We  consider graphons $W$ whose skeleton graphs $S$ are line graphs with a single self-loop attached on one of the ending nodes (note that if there is no self-loop, then by Theorem~\ref{thm:main}, $W$ does not have the $H$-property).  See Fig.~\ref{fig:linegraphon} for illustration.  For convenience, we call such graphons {\em line graphons}.

\begin{proposition}\label{prop:hproplineG}
Let $W$ be a line graphon (so Condition~1 is satisfied). If Condition~2B is satisfied, then $W$ has the $H$-property.   
\end{proposition}

To establish Proposition~\ref{prop:hproplineG}, we first express the incidence matrix $Z$ of the skeleton graph $S$ as follows: 
\begin{equation}\label{eq:Zline}
Z=\frac{1}{2} \begin{bmatrix}
1 & 0 & 0 & 0 &\cdots & 0 & 0 & 0\\
1 & 1 & 0 & 0 & \cdots & 0 & 0 & 0\\
0 & 1 & 1 & 0 & \cdots & 0 & 0 & 0\\
0 & 0 & 1 & 1 &  & 0 & 0 & 0\\
\vdots & \vdots &   &  & \ddots  &  &\vdots  & \vdots\\
0 & 0 & 0 & 0 &  & 1 & 0 & 0\\
0 & 0 & 0 & 0 &  & 1 & 1 & 0\\
0 & 0 & 0 & 0 & \cdots & 0  & 1 & 2\\
\end{bmatrix}
\end{equation}
We need the following lemma:

\begin{lemma}\label{lem:xiinequalities}
If $x=(x_1,\ldots,x_q) \in \rint \cX(S)$, then the entries of $x$ satisfy the following inequalities: 
\begin{equation}\label{eq:ineqline}
\sum_{\ell = 0}^{k-1} (-1)^\ell x_{k - \ell} > 0, \quad \forall k = 1,\ldots, q. 
\end{equation}
\end{lemma}

\begin{proof}
Since $x\in \rint \cX(S)$, one can write $x = \sum^q_{\ell = 1} \alpha_\ell z_\ell$ where $z_\ell$ is the $\ell$th column of the matrix $Z$ in~\eqref{eq:Zline}, and $ 0 < \alpha_\ell < 1$, for all $\ell = 1, \ldots, q$. In particular, 
$$
x_\ell = \frac{1}{2}
\begin{cases}
\alpha_1 & \mbox{if } \ell = 1, \\
\alpha_{\ell - 1} + \alpha_\ell & \mbox{if } 1 < \ell < q, \\
\alpha_{q-1} + 2\alpha_q & \mbox{if } \ell = q.
\end{cases}
$$
It then follows that for any $k = 1,\ldots, q$,
$$
\sum_{\ell = 0}^{k-1} (-1)^\ell x_{k - \ell} = 
\frac{1}{2}
\begin{cases}
\alpha_k & \mbox{if } 1\leq k < q, \\
2\alpha_q & \mbox{if } k = q,
\end{cases}
$$
which is positive. 
This establishes~\eqref{eq:ineqline}.
\end{proof}

We can now prove Proposition~\ref{prop:hproplineG}:

\begin{proof}[Proof of Proposition~\ref{prop:hproplineG}]
Recall that for a given $G_n\sim W$, $n_i = |\pi^{-1}(u_i)|$, where $\pi$ is defined in~\eqref{eq:defpi}. On the one hand, using~\eqref{eq:oldlemma1}, we have that $n_i/n$ converges to $x_i$ as $n \to \infty$. On the other hand, from Lemma~\ref{lem:xiinequalities}, we have the inequalities~\eqref{eq:ineqline}. These two facts imply that almost surely as $n\to \infty$, we have  
\begin{equation}\label{eq:ninequalities}
    \sum_{\ell = 0}^{k-1} (-1)^\ell n_{k - \ell} > 0, \quad \forall k = 1,\ldots, q. 
\end{equation}
Thus, in the sequel, we can assume that the above inequalities are satisfied.  
We  claim that $\vec G_n$ admits a Hamiltonian decomposition almost surely as $n\to \infty$. 
If the claim is true, then $W$ has the $H$-property.  

We now proceed with the proof of the claim. 
For convenience, let $V_i:= \pi^{-1}(u_i)$ for $i = 1,\ldots, q$. 
To this end, consider the subgraph $G^{1,2}_n$ of $G_n$ induced by $V_1 \cup V_2$. Let $\sigma = (\sigma_0,\ldots,\sigma_q)$ be the partition for $W$ and $p^{1,2}$ be the value of $W$ over the rectangle $[\sigma_0,\sigma_1)\times [\sigma_1,\sigma_2)$. Note that $p^{1,2}$ is strictly positive (because otherwise there will be no edge $(u_1,u_2)$ in the skeleton graph). 
By construction,  it should be clear that $G^{1,2}_n$ is an Erd\H{o}s-R\'enyi random bipartite graph $B(n_1,n_2,p^{1,2})$. 
Denote by ${\cal E}^{1,2}_n$ the event that $G^{1,2}_n$ admits a left-perfect matching. 
Because $n_1 < n_2$ by~\eqref{eq:ninequalities} and because $n_1\to \infty$ as $n\to \infty$ (since $n_1/n$ converges to $x_1 > 0$), we know from Lemma~\ref{lem:erdosperfectmatching} that ${\cal E}^{1,2}_n$ is true almost surely as $n\to \infty$. In the sequel, we condition  on the event ${\cal E}^{1,2}_n$ and fix a left-perfect matching $P^{1,2}_n$  in $G^{1,2}_n$.
Let $G'^{1,2}_n$ be the subgraph of $G^{1,2}_n$ induced by $P^{1,2}_n$ (more precisely, induced by the nodes incident to edges in $P^{1,2}_n$). Then, by construction, $P^{1,2}_n$ is a {\em perfect matching} of $G'^{1,2}_n$. 
As argued in Subsection~\ref{ssec:er}, if we let $\vec P^{1,2}_n$ be the subset of edges in $\vec G'^{1,2}_n$ obtained by replacing every undirected edge in $P^{1,2}_n$ with two oppositely directed edges, then $\vec P^{1,2}_n$ gives rise to a Hamiltonian decomposition of $\vec G'^{1,2}_n$ which is comprised only of two-cycles.

Denote by $V'_2$ the set of nodes in $V_2$ that are {\em not} incident to edges in $P^{1,2}_n$. Let $n_2':= |V'_2| = n_2 - n_1 > 0$. 
Similarly, define the subgraph $G^{2,3}_n$ of $G_n$ induced by $V'_2\cup V_3$. It is an Erd\H{o}s-R\'enyi random bipartite graph $B(n'_2,n_3,p^{2,3})$ where $p^{2,3}> 0$ is the value of $W$ over the rectangle $[\sigma_1,\sigma_2)\times [\sigma_2,\sigma_3)$. By~\eqref{eq:ninequalities}, we have that $n_3 - n'_2 = n_3 - n_2 + n_1 > 0$. Let ${\cal E}^{2,3}_n$ be the event that $G^{2,3}_n$ admits a left-perfect matching. Using the same arguments as above, we know that ${\cal E}^{2,3}_n$ is true almost surely as $n\to \infty$. Fix a left-perfect matching $P^{2,3}_n$ of $G^{2,3}_n$. 
Let $G'^{2,3}_n$ be the subgraph of $G^{2,3}_n$ induced by $P^{2,3}_n$, which admits $P^{2,3}_n$ as a {\em perfect matching}. Consequently, $\vec P^{2,3}_n$ yields a Hamiltonian decomposition of $\vec G'^{2,3}_n$.

One can repeat the above arguments as follows: Given a left-perfect matching $P^{k-1,k}_n$, for $1\leq k \leq q-1$, we let $V'_k$ be the subset of $V_k$ such that nodes in $V'_k$ are {\em not} incident to the edges in the left-perfect matching $P^{k-1,k}$. 
We have that $n'_k :=|V'_k|= \sum_{\ell = 0}^{k-1} (-1)^\ell n_{k - \ell} > 0$ and it follows from~\eqref{eq:ninequalities} that $n'_k < n_{k+1}$.   
We then consider the subgraph $G^{k,k+1}_n$ of $G_n$ induced by $V'_k\cup V_{k+1}$, which is a random bipartite graph\footnote{For the case $k = q-1$, the subgraph $G^{q-1,q}_n$ {\em contains} a random bipartite graph, and additional edges are added randomly between nodes of $V_q$ following an Erd\H{o}s-R\'enyi procedure. Note that edges that can appear in the bipartite graph and the ones that can appear between nodes of $V_q$ are distinct and, hence, their appearances are independent.} $B(n'_k,n_{k+1},p^{k,k+1})$ with $p^{k,k+1}$ strictly positive. 
Then, the event  ${\cal E}^{k,k+1}_n$ that  $G^{k,k+1}_n$ admits a left-perfect matching $P^{k,k+1}_n$ is true almost surely. 
Conditioning upon this, we fix a left-perfect matching $P^{k,k+1}_n$ of $G^{k,k+1}_n$ and let  $G'^{k,k+1}_n$ be the subgraph of $G^{k,k+1}_n$ induced by $P^{k,k+1}_n$. It admits $P^{k,k+1}_n$ as a {\em perfect matching}. Then,  $\vec P^{k,k+1}_n$ yields a Hamiltonian decomposition of $\vec G'^{k,k+1}_n$.

Now, let $V'_q$ be the subset of $V_q$ whose nodes are {\em not} incident to the edges in the left-perfect matching $P^{q-1,q}$ and denote by $G^{q}_n$ of the subgraph $G_n$ induced by $V'_q$. 
First, note that $n'_q:=|V'_q| = \sum^{q-1}_{\ell = 0} (-1)^\ell n_{q - \ell}$, which is strictly positive by~\eqref{eq:ninequalities}. In fact, since $n_i/n\to x_i$ as $n\to\infty$,
we have that $n'_q/n \to \sum^{q-1}_{\ell = 0} (-1)^\ell x_{q - \ell} > 0$. 
In particular, it holds that $n'_q \to \infty$ almost surely as $n\to \infty$. 
Next, note that $G^q_n$ is an Erd\H{o}s-R\'enyi random graph $R(n'_q,p^{q,q})$, with $p^{q,q}>0$, where $p^{q,q}$ is the value of $W$ over the square $[\sigma_{q-1},\sigma_q]^2$. 
It then follows from Corollary~\ref{cor:erhamiltoniandec} that $G^q_n$ admits a Hamiltonian decomposition almost surely as $n'_q\to \infty$.

Finally, we conclude this proof by noting that the subgraphs $\vec G'^{1,2}_n,\ldots,\vec G'^{q-1,q}_n$ and $\vec G^q_{n}$ of $\vec G_n$ are disjoint and cover all nodes of $\vec G_n$. Moreover, each subgraph admits a Hamiltonian decomposition. The cycles in these decompositions are thus all disjoint and cover every  node of $\vec G_n$. Together, they form a Hamiltonian decomposition of $\vec G_n$.
\end{proof}

\begin{figure}\label{fig:fish}
\centering
\includegraphics{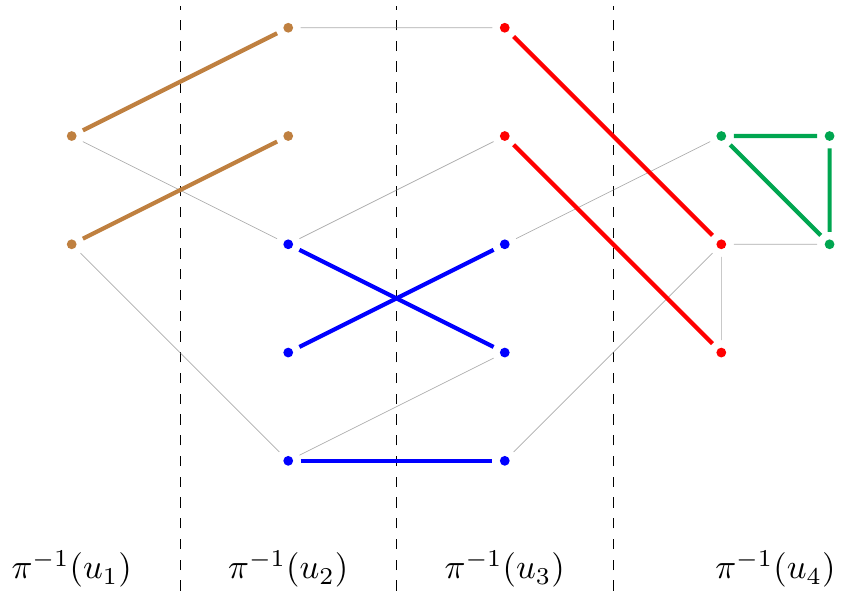}
\caption{Illustration of the proof of Proposition~\ref{prop:hproplineG}: The graph $G_n$ is sampled from the line graphon $W$ illustrated in Fig.~\ref{fig:linegraphon}. The nodes in each $\pi^{-1}(u_i)$, for $1\leq i\leq 4$, are placed in the correspondingly labelled columns. The subgraph in brown corresponds to $G'^{1,2}_n$ in the proof which admits a perfect matching. The subgraphs in blue, red, and green correspond to  $G'^{2,3}_n$,  $G'^{3,4}_n$, and $G^{4}_n$, respectively.}
\end{figure}

\subsection{When is the $H$-property not a zero-one property?}\label{ssec:notonezero}

In this subsection, we study a ``borderline'' case illustrating that the $H$-property is not zero-one for {\em all} step-graphons. To this end, consider the following step-graphon:
\begin{equation}\label{eq:Wborderline}
W(s,t) = 
\begin{cases}
0 & \mbox{if } 0 \leq s, t < 0.5, \\
p & \mbox{otherwise},
\end{cases}
\end{equation}
where $p\in (0,1]$. 
See Fig.~\ref{sfig1:stepgraphonborder} for illustration. 
This graphon satisfies Conditions~1 and 2A, but does {\em not} satisfy Condition 2B. Indeed, the incidence matrix  of its skeleton graph is given by
\begin{equation}\label{eq:zborderline}
Z=\frac{1}{2}\begin{bmatrix}1 & 0 \\ 1 & 2
\end{bmatrix}, 
\end{equation}
so the edge polytope $\cX(S)$ is a line segment in $\R^2$ with the ending points $(0.5.0.5)$ and $(0,1)$. However, the associated concentration vector $x^*$ is given by $(0.5,0.5)$, which is not in the interior of $\cX(S)$. We now have the following result: 

\begin{figure}[t]
    \centering
    \subfloat[\label{sfig1:stepgraphonborder}]{
\includegraphics{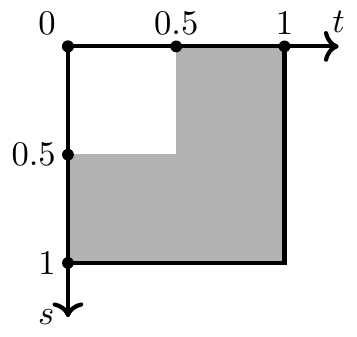}
}
\qquad
\subfloat[\label{sfig1:skeletonborder}]{
  \includegraphics{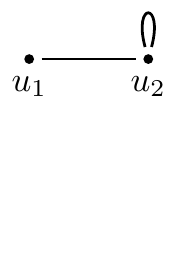}
}
\caption{ {\em Left:} The step-graphon $W$ given in~\eqref{eq:Wborderline},  which has the partition $\sigma=(0,0.5,1)$. The step-graphon takes the values $0$ (in white) and $0 < p  \leq 1$ (in grey). {\em Right:} Its associated skeleton graph.}
    \label{fig:borderlinecase}
\end{figure}

\begin{proposition}\label{prop:borderline}
Let $G_n \sim W$ for the step-graphon $W$ given in~\eqref{eq:Wborderline}. Then,
$$
\lim_{n\to\infty}\mathbb{P}(\vec G_n \mbox{ has a Hamiltonian decomposition}) = 0.5.
$$
\end{proposition}

\begin{proof}
Recall that for $G_n \sim W$,  we have set $n_i=|\pi^{-1}(u_i)|$, for $i = 1,2$. 
Now, consider two cases: 
\vspace{.1cm}

\noindent
{\em Case 1: $n_1 - n_2 > 0$.} First, note that the probability of occurrence of such a case is $\frac{1}{2}$. We next show that in this case, $\vec G_n$ cannot admit a Hamiltonian decomposition. Suppose, to the contrary, that $\vec H$ is a Hamiltonian decomposition in $\vec G_n$. 
Let $v_1$ be an arbitrary node in $\pi^{-1}(u_1)$. Consider the cycle $\vec C \in \vec H$ to which $v_1$ belongs. We express $\vec C$ as a sequence of nodes $\vec C=v_1\cdots v_kv_1$. Because the $n_1$ nodes in $\pi^{-1}(u_1)$ do not have any edge between them, it is clear that no two adjacent nodes in $\vec C$ can belong to $\pi^{-1}(u_1)$. It then follows that the number of nodes of $\vec C$ belonging to $\pi^{-1}(u_1)$ is less than or equal to $|\vec C|/2$. In particular, it implies that $$|\pi^{-1}(u_1) \cap \vec C | \leq |\pi^{-1}(u_2) \cap \vec C |.$$ 
This holds for all cycles in $\vec H$. But since these cycles are disjoint and cover all the nodes of $\vec G_n$, we have to conclude that $n_1\leq n_2$, which is  a contradiction.

\vspace{.1cm}

\noindent
{\em Case 2: $n_1 - n_2 < 0$.} The probability of the occurrence of this case is also $\frac{1}{2}$. 
Consider the subgraph $B$ of $G_n$ obtained by removing the edges between nodes of $\pi^{-1}(u_2)$. By construction, it is a random bipartite graph $B(n_1,n_2,p)$ with $n_1 < n_2$.  Moreover, since $n_1/n \to 1/2$ as $n\to\infty$, it is almost sure that $n_1\to\infty$ as $n\to\infty$. We can then apply Lemma~\ref{lem:erdosperfectmatching} to $B$ and conclude that it contains a left-perfect matching $P$ almost surely. 
Similarly, as done in the proof of Proposition~\ref{prop:hproplineG}, we consider the subgraph $G'_n$ of $G_n$ induced by $P$; the edges in $\vec P$ form a Hamiltonian decomposition of $\vec G'_n$ which is comprised of all two-cycles. 
We next consider the subgraph $G''_n$ of $G_n$ induced by the nodes in $\pi^{-1}(u_2)$ that are {\em not} incident to $P$. 
Then, clearly, $G'_n$ and  $G''_n$ are disjoint and they together cover all the nodes of $G_n$. It thus suffices to show that $\vec G''_n$ admits a Hamiltonian decomposition almost surely to complete the proof of case 2. 
To establish this fact, note that $G''_n$ is an \er random graph on $(n_2-n_1)$ nodes with parameter $p$. We now claim that $(n_2-n_1) \to \infty$ as $n \to \infty$.

To see this, let $X_i$ be the random variable defined as follows: $X_i=1$ if  node $i$ belongs to $\pi^{-1}(u_1)$ and $X_i=-1$ if node $i$ belongs to $\pi^{-1}(u_2)$. Following the sampling procedure given in Section~\ref{sec:Hproperty}, it should be clear that for the $W$ as in~\eqref{eq:Wborderline}, the $X_i$'s are independent, identically distributed and follow a Bernoulli distribution with parameter $\frac{1}{2}$.
We now define their normalized cumulative sum $$\tau_n := \frac{1}{\sqrt{n}}\sum_{i=1}^n X_i =  \frac{n_1 - n_2}{\sqrt{n}};$$ by the central limit theorem~\cite{durrett2019probability}, $\tau_n$ converges in law to a normal random variable $\tau \sim N(0,1)$ as $n\to\infty$. Consequently, it is almost sure that as $n \to \infty$,
\begin{equation*}\label{eq:largediff}
|n_1 - n_2| > \log n,   
\end{equation*}
which proves the claim. 

Finally, by combining the claim with Corollary~\ref{cor:erhamiltoniandec}, we conclude that $\vec G''_n$ admits a Hamiltonian decomposition almost surely as $n\to\infty$.
\end{proof}


\section{Conclusions}\label{sec:conclusions}
We have established in this paper the sufficiency of conditions 1 and 2B given in Subsection~\ref{ssec:condHprop} for line graphons to have the $H$-property. We have also illustrated the importance of the distinction between conditions 2A---which requires the concentration vector to belong to the edge polytope $\cX(S)$ of the skeleton graph $S$--- and condition 2B---which requires the concentration vector to belong to the {\em relative interior} of $\cX(S)$. While condition $2A$ is necessary, it is condition 2B which is sufficient. We have illustrated this fact by exhibiting a graphon which satisfied conditions~1 and~2A, but not~2B, and shown that graphs sampled from these graphons admitted Hamiltonian decompositions with probability $1/2$ asymptotically for $n \to \infty$.

\bibliographystyle{amsplain}      
\bibliography{refs.bib}       
\end{document}